\documentclass[a4paper,12pt,reqno]{amsart}
\usepackage{tikz}
\usetikzlibrary{arrows,shapes}
\usepackage{amsmath}%
\usepackage{amsfonts}%
\usepackage{amssymb}%
\usepackage{graphicx}
\usepackage{float} 
\usepackage{subfigure} 
\usepackage{lineno,hyperref}
\hypersetup{
	colorlinks = true,
	linkcolor = blue,
	anchorcolor = blue,
	citecolor = blue,
	filecolor = blue,
	urlcolor = blue
}
\newtheorem{thm}{Theorem}
\theoremstyle{plain}

\newtheorem{coro}{Corollary}

\newtheorem{lemm}{Lemma}

\numberwithin{equation}{section}

\newcommand{\C}{\mathbb{C}}

\begin{document}

\title[Hayman conjecture]{A note on Hayman's problem}
\author{Jiaxing Huang $^1$}
\address{$^1$College of Mathematics and Statistics, Shenzhen University, Guangdong, 518060, P. R. China}
\address{$^2$Institute of Mathematics, Academy of Mathematics and Systems Sciences,
	Chinese Academy of Sciences, Beijing 100190, China}
\email{wangyf@math.ac.cn}
\author{Yuefei Wang $^{1, 2}$}
\thanks{Corresponding author: Yuefei Wang.}

\begin{abstract} In this note, it is shown that the differential polynomial of the form $Q(f)^{(k)}-p$ has infinitely many zeros, and particularly  $Q(f)^{(k)}$ has infinitely many fixed points for any positive integer $k$, where $f$ is a transcendental meromorphic function, $p$ is a nonzero polynomial and $Q$ is a polynomial with coefficients in the field of small functions of $f$. The results are traced back to Problem 1.19 and Problem 1.20 in the book of research problems by Hayman and Lingham \cite{HL21}. As a consequence, we give an affirmative answer to an extended problem on the zero distribution of $(f^n)'-p$, proposed by Chiang and considered by Bergweiler \cite{Ber97}. Moreover, our methods provide a unified way to study the problem of the zero distributions of partial differential polynomials of meromorphic functions in one and several complex variables with small meromorphic coefficients.
\end{abstract}
\maketitle

\section{Introduction and main results}
In this paper, we focus on the zero distributions of (partial) differential polynomials in a meromorphic function $f$ with small meromorphic coefficients.  We assume that the reader is familiar with the standard notations and some basic results in Nevanlinna theory  (see \cite{Hay64,IIpo11}).

In 1959, Hayman \cite{Hay59, HL21} conjectured that if $f$ is a transcendental meromorphic function and $n\geq 2$ is an integer, then $(f^n)'$ assumes every nonzero complex number infinitely often. He proved this conjecture for $n\geq 4$. The case for $n=3$ was settled by Mues \cite{Mue79} in 1979, and the remaining for $n=2$ was obtained independently by Bergweiler and Eremenko \cite{BE95}, Chen and Fang \cite{CF95} and Pang and Zalcman \cite{PZ99}. One principal extension was studied in \cite{BE95}, the authors showed the Hayman conjecture is valid if $(f^n)'$ is replaced by $(f^n)^{(k)}$ for $n> k\geq 1$.

Two related questions arising in connections with Hayman's problem are as follows.

The first one is to consider the zero distribution of $(f^n)'-p$ where $p$ is a small function of $f$. In fact, this extension was originated from the study of the zero distribution of $ff'-p$, proposed by Yik-Man Chiang in \cite{Ber97}. There do exist some partial results but it remains open. For example, Zhang\cite{Zha94} proved that if the deficiency $\delta(\infty, f)$ of the poles of $f$ is strictly greater than $7/9$, then $ff'-p$ has infinitely many zeros. Bergweiler \cite{Ber97}  gave a positive answer when $p$ is a nonzero polynomial and $f$ is of finite order.
\begin{thm}[Theorem of \cite{Ber97}]\label{thm:ber}
	Let $f$ be a transcendental meromorphic function of finite order and $p$ be a polynomial which does not vanish identically. Then $ff'-p$ has infinitely many zeros.
\end{thm}
The second question given by Eremenko and Langley \cite{EL08} is whether one can consider a more general differential polynomial of $f$ such as a linear differential polynomial $$F:=f^{(k)}+a_{k-1}f^{(k-1)}+\dots+a_0f$$ with suitable small function coefficients $a_j$ instead of $f^{(k)}$ only.  Very recently, An and Phuong \cite{AP20} investigated this question for a differential polynomial $Q(f)$ with some restrictive conditions.
\begin{thm}[Theorem 1 of \cite{AP20}]\label{thm:AP}
	Let $f$ be a transcendental meromorphic function, and $Q(z)=b(z-a_1)^{m_1}(z-a_2)^{m_2}\cdots(z-a_l)^{m_l}$ be a polynomial of degree $q$, where  $b\in\C^*$ and $a_1, \dots, a_l\in\C$. If $q\geq l+1$, then $(Q(f))^{(k)}$ takes every finite nonzero value infinitely often, for any positive integer $k$.
\end{thm}
Inspired by Theorems \ref{thm:ber} and \ref{thm:AP}, we generalize An and Phuong's result. In fact, we consider more general situations that the coefficients of the polynomial $Q$ are allowed to be functions meromoprhic in $\C$, and the nonzero value is replaced by a polynomial. Our main result is the following.

\begin{thm}\label{thm:H}
	Let $f$ be a transcendental meromorphic function, $p$ be a nonzero polynomial and $Q(z)=b(z-a_1)^{m_1}(z-a_2)^{m_2}\cdots(z-a_l)^{m_l}$ be a polynomial of degree $q$, where  $b\not\equiv 0, a_1, \dots, a_l$ are small functions of $f$. If $q\geq l+1$, then $(Q(f))^{(k)}-p$ has infinitely many zeros, and particularly $(Q(f))^{(k)}$ has infinitely many fixed points, for any positive integer $k$.
\end{thm}

As a consequence of Theorem \ref{thm:H}, we give an affirmative answer to the question of Chiang when $p$ is a polynomial, without growth restriction on $f$.

\begin{coro}[Hayman's problem for polynomials]\label{cor:Hay} If $f$ is a transcendental meromorphic function and $p$ is a nonzero polynomial, then $f^nf'-p$ has infinitely many zeros, for any positive integer $n$.
\end{coro}

The corollary follows immediately if one takes $k=1$, and $Q(z)=z^{n+1}$. Moreover, when $Q(z)=z^n$, $n\geq 2$,  one can extend a result of Hayman in (\cite{Hay59}, Theorem 2). Note that Hayman considered the value distribution of $(f^n)^{(k)}-c$ for any nonzero complex number $c$, while we can take $c$ to be a polynomial.
\begin{coro}
	Let $f$ be a transcendental meromorphic function and $p$ be a nonzero polynomial. Then for $n\geq 2, k\in\mathbb{N}$,  $(f^n)^{(k)}-p$ has infinitely many zeros.
\end{coro}

As a consequence, we also obtain the following result, which could be regarded as a precursor of Theorem \ref{thm:gh}.

\begin{coro}
	Let $f$ be a transcendental meromorphic function and $P$ be a nonzero polynomial, then $Pf^n-f'$ has infinitely many zeros for any $n\geq 3$.
\end{coro}
\begin{proof}
	Let $f=1/g$, then
	$$Pf^n-f'=P/g^n+g'/g^2=\frac{P+g'g^{n-2}}{g^n}.$$
	Hence the result follows from Corollary \ref{cor:Hay}.
\end{proof}

In view of the above results, we only consider the polynomial case, so what if that of small functions? This question is in general not easy to answer. However, using some ideas from Liao and Ye \cite{LY14}, the classical Logarithmic Derivative Lemma and the Clunie Lemma, we give some partial results as follows.

To describe our result, we need to introduce some classes of meromorphic functions.

Let $T(r, f)$ be the Nevanlinna characteristic function of $f$.
We denote by $S(r, f)$ any quantity  which is of growth $o(T(r, f))$ as $r\rightarrow\infty$ outside a set $E\subset(0, \infty)$ of finite measure. A meromorphic function $y$ is called a \emph{small function} of $f$ if it satisfies that $T(r, y)=S(r, f)$. The family of small functions of $f$ is defined by  $\mathcal{S}_f$.
By $\mathcal{N}_0$ and $\mathcal{S}_0$ we mean that the family of meromorphic functions $y$ with finitely many poles and $\mathcal{S}_0=\mathcal{S}_f\cap \mathcal{N}_0$, respectively. Clearly, the field $\C(z)$ of rational functions and the ring of entire functions are contained in $\mathcal{N}_0$.
\begin{thm}\label{thm:gh}
	Let $$P(z, w)=\sum_{j=m}^nb_j(z)w^j, \quad b_m\neq0$$ be a polynomial in $w$ with coefficients $b_j(z)$ in the family $\mathcal{S}_0$, and $$L(z, w)=\sum_{|I|=0}^ka_I(z)w^{i_0}w'^{i_1}\cdots (w^{(q)})^{i_q}$$ be a differential polynomial in $w$ over $\mathcal{S}_0$, where $I=(i_0, \dots, i_q)$ is a multi-index with length $|I|=i_0+\cdots+i_q$. If $f$ is a transcendental meromorphic function in $\mathcal{N}_0$ such that $L(z, f)\not\equiv 0$,  then $P(z, f)+L(z, f)$ has infinitely many zeros for any $m\geq k+2$.
\end{thm}
In the special case that $P(z, w)=Q_1(z)w^n$ and $L(z, w)=Q_2(z)w^{(q)}$ where $Q_1(z)$ and $Q_2(z)$ are nonzero rational functions, $n, q\in\mathbb{N}$, we recover some known results for entire functions in \cite{Hay59}.
\begin{coro}
	Let $f$ be a transcendental entire function and let $R$ be a rational function. If $n\geq 3$, then $Q_1f^n+Q_2f^{(q)}-R$ has infinitely many zeros, and hence $Q_1f^n-Q_2f^{(q)}$ has infinitely many fixed points for all $q\in\mathbb{N}$.
\end{coro}

 Our methods also provide a unified way to study Hayman's problem of the zero distribution of a partial differential polynomial of meromorphic functions of several complex variables with small meromorphic coefficients.

Let $I=(i_1, \dots, i_N)\in\mathbb{N}^{N}$ be an integer multi-index. The partial derivative of a meromorphic function $f$ in $(z_1, \dots, z_N)$ is defined by $$D^If=\frac{\partial^{|I|}f}{\partial z^I}=\frac{\partial^{|I|}f}{\partial z_1^{i_1}\cdots\partial z_N^{i_N}},\ |I|=\sum_{j=1}^Ni_j.$$

\begin{thm}\label{thm:ghs}Let $I_{\mu}=(i_{1\mu}, \dots, i_{N\mu})\in\mathbb{N}^{N}$ be a multi-index for $\mu=1,\cdots, s$.
	Let $$P(z, w)=\sum_{j=m}^nb_j(z)w^j, \quad b_m\neq0$$ be a polynomial in $w=w(z_1, \dots, z_N)$ with meromorphic coefficients $b_j(z)$, and $$L(z, w)=\sum_{\mathbf{l}\in\mathcal{L}}a_{\mathbf{l}}(z)w^{l_0}(D^{I_1}w)^{l_1}\cdots(D^{I_s}w)^{l_s}, \quad \mathbf{l}=(l_0,\dots, l_s)\in\mathbb{N}^{s+1}$$ be a partial differential polynomial in $w$ with total degree at most $k$, where $\mathcal{L}$ is a finite set of distinct elements and $b_j(z)$'s, $a_{\mathbf{l}}(z)$'s are meromorphic functions on $\C^N$ with an algebraic pole set. Let $f$ be a transcendental meromorphic function on $\C^N$ with an algebraic pole set such that $L(z, f)\not\equiv 0$. If $a_{\mathbf{l}}(z)$'s are small functions of $f$,  then the zero set of $P(z, f)+L(z, f)$ must be transcendental for any $m\geq k+2$. \end{thm}

Here, a set is said to be {\it algebraic} if it is contained in the zero set of a polynomial, otherwise, it is said to be {\it transcendental}.

\vskip.2in

\noindent\textbf{1.2\ Organization of the paper.} The rest of the paper is organized as follows. In Sect. 2 we state several results
that will be used in our proofs. Then we prove our main result (Theorems \ref{thm:H}, \ref{thm:gh} and \ref{thm:ghs}) in Sect. 3.

\section{Some Lemmas}
We first recall some useful lemmas.
\begin{lemm}[Theorem 2.2.5 of \cite{IIpo11}, FMT for small functions]\label{lem:FMT}
	Let $f$ be a meromorphic function in $\mathbb{C}$ and $a\in\mathcal{S}_f$, then $$T(r, a, f)=T(r, f)+S(r, f)$$ and $$T(r, af)=T(r, f)+S(r, f), \quad a\neq 0.$$
\end{lemm}
In 2013, Fang and Wang proved the following proposition, which plays the crucial role in the proof of Theorem \ref{thm:H}.
\begin{lemm}[Proposition 3 of \cite{FW13}]\label{pro:FW}
	Let $g$ be a transcendental meromorphic function in $\mathbb{C}$, $k$ be a positive integer, and $p(\not\equiv 0)$ be a polynomial. Then for any $\epsilon>0$,
	\begin{equation}\label{eqn:FW}
		(1-\epsilon)T(r, g)+(k-1)\overline{N}(r, g)\leq\overline{N}(r, \frac{1}{g})+N(r, \frac{1}{g^{(k)}-p})+S(r, f).
	\end{equation}	
\end{lemm}

\begin{lemm}[Theorem 2.3.3 of \cite{IIpo11}]\label{lem:LDL}
	Let $f$ be a transcendental meromorphic function and $k\geq 1$ be an integer. Then $$m\left(r, \frac{f^{(k)}}{f}\right)=S(r, f).$$
\end{lemm}
\begin{lemm}[Lemma 3.3 of \cite{Hay64}, Clunie's lemma]\label{lem:Cl}
	Let $f$ be a transcendental meromorphic function in the complex plane such that $$f^nP(z, f)=Q(z,f),$$ where $P(z,f)$ and $Q(z,f)$ are polynomials in $f$ and its derivatives with meromorphic coefficients, say $\{a_{\lambda}|\lambda\in\Lambda\}$, such that $T(r, a_{\lambda})=S(r, f)$ for all $\lambda\in\Lambda$. If the total degree of $Q(z, f)$ as a polynomial in $f$ and its derivative is at most $n$, then
	$$m(r, P(z,f))=S(r, f)\quad \mathrm{as}\quad r\rightarrow\infty.$$
	
\end{lemm}

\section{Proof of Theorems \ref{thm:H}, \ref{thm:gh} and \ref{thm:ghs}}
\begin{proof}[Proof of Theorem \ref{thm:H}]
	Applying Lemma \ref{pro:FW} to $g=Q(f)$, we have
	\begin{eqnarray*}
		&\ &(1-\epsilon)T(r, Q(f))+(k-1)\overline{N}(r, Q(f))\\
		&\leq &\overline{N}\left(r, \frac{1}{Q(f)}\right)+N\left(r, \frac{1}{Q(f)^{(k)}-p}\right)+S(r, f)\\
		&\leq&\sum_{i=1}^l\overline{N}\left(r, \frac{1}{f-a_i}\right)+\overline{N}\left(r, \frac{1}{b}\right)+N\left(r, \frac{1}{Q(f)^{(k)}-p}\right)+S(r, f)\\
		&\leq&lT(r, f)+N\left(r, \frac{1}{Q(f)^{(k)}-p}\right)+S(r, f),
	\end{eqnarray*}
	and then from the FMT for small functions (Lemma \ref{lem:FMT}), it follows that  $$(q-l-\epsilon)T(r, f)\leq N\left(r, \frac{1}{Q(f)^{(k)}-p}\right)+S(r, f)$$ for any $\epsilon>0$. Therefore, $Q(f)^{(k)}=p$ has infinitely many roots when $q\geq l+1$, which means that $Q(f)^{(k)}-p$ has infinitely many zeros.
\end{proof}

\begin{proof}[Proof of Theorem \ref{thm:gh}]
	Let $P(f)=P(z, f)$ and $L(f)=L(z, f)$. Suppose that $P(f)+L(f)$ takes zero finitely many. As $f$ has finitely many poles and the coefficients of $P$ and $L$ are in $\mathcal{S}_0$, it follows that
	\begin{equation}\label{eqn:d}
		\sum_{j=m}^nb_jf^j+L(f)=Ae^h
	\end{equation} where $A$ is a nonzero rational function and $h$ is an entire function with $T(r, h)=S(r, f)$.
	By differentiating both sides of (\ref{eqn:d}), we obtain that
	\begin{equation}\label{eqn:d1}
		\sum_{j=m}^nB_jf^{j-1}+L(f)'=A^*e^h
	\end{equation} where $B_j=b_j'f+jb_jf'$ is a linear differential polynomial of $f$ over $\mathcal{S}_0$ and $A^*=A'+Ah'$ with $T(r, A^*)=S(r, f)$.
	It follows from (\ref{eqn:d}) and (\ref{eqn:d1}) that
	\begin{equation}\label{eqn:cd}
		f^{m-1}H(z, f)=Q(z, f)
	\end{equation}
	where $$H(z, f)=\sum_{j=m}^n(AB_jf^{j-m}-A^*b_jf^{j-m+1})$$ is a differential polynomial in $f$ with coefficients in $\mathcal{S}_f$ and
	\begin{align}\label{eqn:Q}
		Q(z, f)&=A^2\left(\frac{-L(f)}{A}\right)'+Ah'L(f)
	\end{align} is also an $\mathcal{S}_f$--differential polynomial with total degree at most $k(<m-1)$.
	We claim that $H(z, f)\not\equiv 0$. Otherwise, in view of (\ref{eqn:cd}), (\ref{eqn:Q}) and $L(f)\not\equiv0$, one has $Q(z, f)\equiv 0$, and then $KL(f)=Ae^h$ with some constant $K$. Since $f$ is a transcendental meromorphic function and $b_j\in\mathcal{S}_0$, (\ref{eqn:d}) gives that $K\neq 1$ and
	$$f^m\left(\sum_{j=m}^nb_jf^{j-m}\right)=(K-1)L(f)$$ with $\deg L=k< m$.
	By Clunie's lemma (Lemma \ref{lem:Cl}), we have
	$$m\left(r, \sum_{j=m}^nb_jf^{j-m}\right)=S(r, f)$$
	Thus, by $N(r, f)=O(\log r)$ and $b_j\in\mathcal{S}_0$, we have $$T\left(r, \sum_{j=m}^nb_jf^{j-m}\right)=S(r, f)$$ yielding a contradiction. Hence $H(z, f)\not\equiv 0$.
	
	Applying Clunie's lemma (Lemma \ref{lem:Cl}) to (\ref{eqn:cd}), we have
	$$m(r, H(z, f))=S(r, f).$$
	As $A^*, A\in\mathcal{S}_f$, $B_j$ is differential polynomials of $f$ with coefficients in $\mathcal{S}_0$ and $b_j\in\mathcal{S}_0$, for $j=m, \dots, n$, it is not hard to see that
	$$N(r, H(z, f))=S(r, f)\quad\mathrm{and\ hence}\quad T(r, H(z, f))=S(r, f).$$
	Therefore, by the FMT of small functions and $T(r, H(z, f))=S(r, f)$,  we have
	$$T(r, Q(z, f))=T(r, H(z, f)f^{m-1})=(m-1)T(r, f)+S(r, f).$$
	From the form of $Q(z, f)$, one can rewrite it as follows
	$$Q(z, f)=\sum_{l=0}^kC_l(z)f(z)^l$$ where $C_l(z)$'s are differential polynomials in $f'/f$ and its derivatives with meromorphic coefficients in $\mathcal{S}_f$, and hence by the LDL, $m(r, C_l)=S(r, f)$ for all $l$ .
	Since
	\begin{align*}
		m\left(r, \sum_{l=0}^kC_lf^l\right)&\leq m\left(r, f\sum_{l=1}^kC_lf^{l-1}\right)+m(r, C_0)+O(1)\\
		&\leq m(r, f)+m\left(r, \sum_{l=1}^kC_lf^{l-1}\right)+m(r, C_0)+O(1),
	\end{align*}
	an immediate inductive argument implies that
	$$m\left(r, \sum_{l=0}^kC_lf^l\right)\leq km(r, f)+\sum_{l=0}^km(r, C_l)+S(r, f)=km(r, f)+S(r, f).$$
	From the form of $Q(z, f)$ in (\ref{eqn:Q}) and $N(r, f)=O(\log r)$, it follows that $$N(r, Q(z, f))=S(r, f),$$ thus one can obtain that $$T(r, Q(z, f))=m(r, Q(z, f))+N(r, Q(z, f))\leq km(r, f)+S(r, f)$$
	and hence $$(m-1)T(r, f)+S(r, f)\leq km(r, f)+S(r, f)\leq kT(r, f)+S(r, f),$$ which is impossible, as $m\geq k+2$.
\end{proof}

\begin{proof}[Proof of Theorem \ref{thm:ghs}]
Together with the $\C^N$-versions of logarithmic derivative lemma (see \cite{Li11,Vit77}) and Clunie's lemma \cite{Li96}, our methods still work. We only need to make slight modifications for the several complex variables cases. We omit the details of the proof.

\end{proof}

\vskip.1in









\end{document}